\documentclass{amsart}

\usepackage{amsmath,amssymb,amsthm,amsfonts,mathrsfs,mathtools}
\usepackage[frame,cmtip,arrow,matrix,line,graph,curve]{xy}
\usepackage{mathabx}
\usepackage[colorlinks,backref=page,citecolor=blue]{hyperref}
\usepackage{tikz}
\usepackage{enumerate}
\usepackage{comment}
\usepackage{enumitem}
\usepackage{todonotes}
\usepackage{tikzit}
\usepackage[capitalise]{cleveref}
\usepackage{bbm}

\tikzstyle{box}=[shape=rectangle, text height=1.5ex, text depth=0.25ex, yshift=0.5mm, fill=white, draw=black, minimum height=5mm, yshift=-0.5mm, minimum width=5mm, font={\small}]
\tikzstyle{gate}=[shape=rectangle, text height=1.5ex, text depth=0.25ex, yshift=0.5mm, fill=white, draw=black, minimum height=5mm, yshift=-0.5mm, minimum width=5mm, font={\small}, tikzit category=circuit]
\tikzstyle{big gate}=[shape=rectangle, text height=1.5ex, text depth=0.25ex, yshift=0.5mm, fill=white, draw=black, minimum height=10mm, yshift=-0.5mm, minimum width=5mm, font={\small}, tikzit category=circuit]
\tikzstyle{Z dot}=[inner sep=0mm, minimum size=2mm, shape=circle, draw=black, fill={rgb,255: red,221; green,255; blue,221}, tikzit category=zx]
\tikzstyle{Z phase dot}=[minimum size=5mm, font={\footnotesize\boldmath}, shape=rectangle, rounded corners=2mm, inner sep=0.2mm, outer sep=-2mm, scale=0.8, tikzit shape=circle, draw=black, fill={rgb,255: red,221; green,255; blue,221}, tikzit draw=blue, tikzit category=zx]
\tikzstyle{X dot}=[Z dot, shape=circle, draw=black, fill={rgb,255: red,255; green,136; blue,136}, tikzit category=zx]
\tikzstyle{X phase dot}=[Z phase dot, tikzit shape=circle, tikzit draw=blue, fill={rgb,255: red,255; green,136; blue,136}, font={\footnotesize\boldmath}, tikzit category=zx]
\tikzstyle{hadamard}=[fill=yellow, draw=black, shape=rectangle, inner sep=0.6mm, minimum height=1.5mm, minimum width=1.5mm, tikzit category=zx]
\tikzstyle{paulibox}=[fill={rgb,255: red,221; green,221; blue,255}, draw=black, shape=rectangle, inner sep=0.6mm, minimum height=5mm, minimum width=5mm, font={\footnotesize}, text height=1.5ex, text depth=0.25ex, tikzit category=zx]
\tikzstyle{vertex}=[inner sep=0mm, minimum size=1mm, shape=circle, draw=black, fill=black, tikzit category=misc]
\tikzstyle{vertex set}=[inner sep=0mm, minimum size=2mm, shape=circle, draw=black, fill=white, font={\footnotesize\boldmath}, tikzit category=misc]
\tikzstyle{small black dot}=[fill=black, draw=black, shape=circle, inner sep=0pt, minimum width=1.2mm, tikzit category=circuit]
\tikzstyle{cnot ctrl}=[fill=black, draw=black, shape=circle, inner sep=0pt, minimum width=1.2mm, tikzit category=circuit]
\tikzstyle{cnot targ}=[fill=white, draw=white, shape=circle, tikzit category=circuit, label={center:$\oplus$}, inner sep=0pt, minimum width=2.1mm, tikzit fill={rgb,255: red,102; green,204; blue,255}, tikzit draw=black]
\tikzstyle{ket}=[fill=white, draw=black, shape=regular polygon, regular polygon sides=3, regular polygon rotate=-30, scale=0.7, inner sep=1pt, tikzit category=circuit, tikzit shape=rectangle, tikzit fill=green]
\tikzstyle{bra}=[fill=white, draw=black, shape=regular polygon, regular polygon sides=3, regular polygon rotate=30, scale=0.7, inner sep=1pt, tikzit category=circuit, tikzit shape=rectangle, tikzit fill=red]
\tikzstyle{scalar}=[shape=rectangle, text height=1.5ex, text depth=0.25ex, yshift=0.5mm, fill=white, draw=black, minimum height=5mm, yshift=-0.5mm, minimum width=5mm, font={\small}]
\tikzstyle{clabel}=[fill=white, draw=none, shape=rectangle, tikzit fill={rgb,255: red,56; green,255; blue,242}, font={\footnotesize}, inner sep=1pt, tikzit category=labels]
\tikzstyle{empty diagram}=[draw={gray!40!white}, dashed, shape=rectangle, minimum width=1cm, minimum height=1cm, tikzit category=misc]
\tikzstyle{white dot}=[Z dot]
\tikzstyle{gray dot}=[X dot]
\tikzstyle{white phase dot}=[Z phase dot]
\tikzstyle{gray phase dot}=[X phase dot]
\tikzstyle{small hadamard}=[hadamard]

\tikzstyle{simple}=[-]
\tikzstyle{hadamard edge}=[-, dashed, dash pattern=on 2pt off 0.5pt, thick, draw={rgb,255: red,68; green,136; blue,255}]
\tikzstyle{box edge}=[-, dashed, dash pattern=on 2pt off 0.5pt, thick, draw={rgb,255: red,203; green,192; blue,225}]
\tikzstyle{brace edge}=[-, tikzit draw=blue, decorate, decoration={brace,amplitude=1mm,raise=-1mm}]
\tikzstyle{diredge}=[->]
\tikzstyle{double edge}=[-, double, shorten <=-1mm, shorten >=-1mm, double distance=2pt]
\tikzstyle{gray edge}=[-, {gray!60!white}]
\tikzstyle{pointer edge}=[->, very thick, gray]
\tikzstyle{boldedge}=[-, line width=1.6pt, shorten <=-0.17mm, shorten >=-0.17mm]
\tikzstyle{implies}=[-implies, double, double distance=2pt]

\input{graph.tikzdefs}

\def\dd{\hbox{-}}

\usepackage[margin=1in]{geometry}

\theoremstyle{definition}

\newtheorem{theorem}{Theorem}[section]

\newtheorem{lemma}[theorem]{Lemma}

\newtheorem{corollary}[theorem]{Corollary}

\newtheorem*{theorem*}{Theorem}
\newtheorem*{lemma*}{Lemma}
\newtheorem*{claim*}{Claim}

\theoremstyle{remark}
\newtheorem{remark}[theorem]{Remark}

\def\FF{\mathbb{F}}

\def\S{\boldsymbol{S}}
\def\E{\boldsymbol{E}}
\def\B{\boldsymbol{B}}

\DeclareMathOperator{\diag}{diag}

\newcommand{\Ones}{\mathbf{1}}
\newcommand{\sm}{\setminus}
\newcommand{\set}[1]{\left\{#1\right\}}

\title{Extremal Graphs for the Lights Out Problem}

\thanks{\fontfamily{cmr}\selectfont \textsc{Princeton University, Princeton, NJ, USA}}
\thanks{\textit{E-mail addresses:} \tt{\{jc3530, sergio.cris, adivoux, varunsiva\}@princeton.edu}}
\author{Julien Codsi}
\author{Sergio Cristancho}
\author{Alexander Divoux}
\author{Varun Sivashankar}

\begin{document}

\maketitle

\begin{abstract}
Lights Out is a game played on a graph $G$ where every vertex has a light bulb that is either on or off, and pressing a vertex $v$ toggles the state of every vertex in the closed neighborhood of $v$. The goal is to find a subset of vertices $S$ such that pressing every vertex in $S$ results in all light bulbs being turned off. We study the extremal graphs for which pressing every vertex is the unique solution to the lights out problem given an initial configuration of all lights on. We show that a graph is extremal if and only if it is even and has an odd number of matchings. Furthermore, there is a bijection between the set of labeled $n$-vertex extremal graphs and the set of symmetric invertible matrices of size $n-2$ over $\FF_2$. We prove that any even graph with no cycle of length $0\pmod 3$ must be extremal. We also demonstrate operations that build larger extremal graphs from smaller ones. Along the way, we prove using the polynomial method that in any even graph, the number of matchings of a fixed size covering an odd subset of vertices is even.
\end{abstract}

\section{Introduction}

Lights Out is a game played on a graph $G$. Each vertex represents a light that is either on or off. Pressing a vertex switches the state of that vertex and all of its neighbors: lights that were on are turned off, and lights that were off are turned on. The objective of the game is to turn all lights off. Pressing a vertex twice has no effect and the vertices can be pressed in any order to achieve the same state, so a sequence of presses is determined entirely by the set of vertices pressed exactly once. Loops and parallel edges do not affect the Lights Out game, so we assume that all graphs are simple unless otherwise stated.

The Lights Out game was first studied by Pelletier \cite{Pelletier87} as Merlin's Magic Square, from the Merlin game console distributed by Parker Brothers in 1978, which was played on a grid graph with 9 vertices. Sutner \cite{Sutner89} \cite{Sutner90} observed that the Lights Out game, or as he called it, the $\sigma$-game, is closely related to the study of additive cellular automata on graphs. The first instance of the name Lights Out in the literature is due to Anderson and Feil \cite{anderson1998turning} after the electronic game by Tiger Electronics from 1995. 

For an arbitrary initial configuration of lights, one may ask whether there is a set of vertices $S$ whose presses turn every light off. In terms of linear algebra, given the adjacency matrix $A_G \in \FF_2^{n\times n}$ and the indicator vector of the initial configuration $p \in \FF_2^n$, this is equivalent to determining whether the system $(A_G + I_n)x = p$ admits a solution over $\FF_2$. Sutner \cite[Thm. 3.1]{Sutner89} proves that an initial configuration $p$ is solvable if and only if $p$ belongs to the orthogonal complement of the kernel of the matrix $A_G+I_n$. This question has been also been studied for specific families of graphs. For example, Anderson and Feil \cite{anderson1998turning} determined the solvable initial configurations for certain grid graphs. For arbitrary graphs, Berman, Borer and Hungerbühler \cite{berman2021lights} gave a linear algebra criterion to determine when a particular state can be achieved from another. It is worth noting that finding a minimal size solution to an arbitrary initial configuration is an NP-complete problem \cite[Thm. 3.2]{Sutner88}. 

In this paper, we focus on the natural setting where all the lights are initially on; the rest of the paper will always assume this initial configuration. Sutner \cite[Thm. 3.2]{Sutner89} proves that this initial configuration always has a solution. This result is equivalent to a theorem of Gallai on partitioning the vertices of a graph into two sets whose induced subgraphs are both even, see for instance \cite{Caro96}. In fact, Gallai's theorem is equivalent to the assertion that the diagonal of any symmetric matrix over $\FF_2$ belongs to its image. Additionally, the solvability of the all lights on configuration can observed in \cite[Thm. 5]{berman2021lights}, which states that any configuration can be inverted. Recently, Mirzaei \cite{mirzaei2025inductive} also gives a proof of this fact using only combinatorial arguments.

We study the class of extremal graphs for the Lights Out game. Given a graph $G$, pressing a subset of vertices $S \subseteq V(G)$ turns all the lights off precisely when $|N[v]\cap S|$ is odd for every $v\in V(G)$. Call such a subset $S$ a \textit{hitting set}. For brevity, we refer to the parity of $|N(v) \cap S|$ as the parity of $v$ when $S$ is clear from context. A hitting set $S$ of $G$ is \textit{proper} if $S$ is a proper subset of $V(G)$. We say that $G$ is \textit{extremal} if it has no proper hitting set. In other words, a graph $G$ is extremal if hitting every vertex of $G$ is the unique solution to the Lights Out game on $G$. For a graph $G$, let $m(G)$ denote the total number of matchings in $G$ (including the empty matching). 

The paper is organized as follows.
\begin{itemize}
\item In \cref{sec-char}, we prove that a graph is extremal if and only if it is even and has an odd number of matchings. 
\item In \cref{sec-count}, we count the number of labeled extremal graphs on $n$ vertices by establishing a bijection to the number of symmetric invertible matrices of size $n-2$.
\item In \cref{sec-0cycle}, we first consider cycles, the simplest even graphs. We show that cycles are extremal if and only if they have length divisible by $3$. We also show that graphs with no cycle of length divisible by 3 must be extremal.
\item In \cref{sec-operations}, we discuss several operations to build larger extremal graphs from smaller extremal graphs. Along the way, we prove an interesting result on the parity of the number of matchings covering an odd set of vertices in even graphs. 
\end{itemize}

\section{Characterization by parity of matchings}\label{sec-char}
The goal of this section is to provide an equivalent definition of extremal graphs that is entirely independent of the Lights Out game, thereby further motivating the study of this family.

 Let $G$ be a graph with $n$ vertices, let $A_G$ be its adjacency matrix, and set $M_G = A_G + I_n$. A first observation is that $G$ is extremal if and only if $M_G \bf{1}_n = \bf{1}_n$ and $M_G$ is invertible over $\FF_2$. With this perspective, we prove the following characterization.

\begin{theorem}\label{thm:oddmatchings}
    A graph $G$ is extremal if and only if $G$ is even and has an odd number of matchings.
\end{theorem}
\begin{proof}
    Consider the matrix $M_G = A_G+I_n$ with coordinates $x_{v,w}\in \FF_2$ labeled by $v,w\in V(G)$. We aim to compute
    \[
        \det(M_G) = \sum\limits_{\sigma\in S_n } \prod_{v\in V(G)} x_{v,\sigma(v)}.
    \]
    Let $\sigma\in S_n$, and let $\sigma=c_1 \dots c_k$ be its unique decomposition into disjoint cyclic permutations of length $\geq 2$. If $c=(v_1 \dots v_\ell)$ is a cyclic permutation, let $\overline{c}=(v_\ell \dots v_1)$ be the reversed cyclic permutation, and define $\overline{\sigma}=\overline{c_1}\dots\overline{c_k}$. Note that $\overline{c}=c$ if and only if $c$ has length $2$, hence $\overline{\sigma}=\sigma$ if and only if $\sigma$ is composed of cycles of length $\leq 2$. 
    
    Now, the term $\prod_{v\in V(G)} x_{v,\sigma(v)}$ is equal to 1 if and only if each $c_i$ forms a cycle in $G$ (here single edges in $G$ are considered cycles of length 2). Furthermore, the term corresponding to $\sigma$ is equal to 1 if and only if the term corresponding to $\overline{\sigma}$ is equal to 1. Therefore
    \[
        \det(M_G) = \sum\limits_{\overline{\sigma}=\sigma } \prod_{v\in V(G)} x_{v,\sigma(v)} + \sum\limits_{\overline{\sigma}\neq \sigma } \prod_{v\in V(G)} x_{v,\sigma(v)} = \sum\limits_{\overline{\sigma}=\sigma } \prod_{v\in V(G)} x_{v,\sigma(v)},
    \]
    since the summation over all $\sigma\neq \overline{\sigma}$ must be zero by our previous observation. 
    
    Finally, if $\sigma$ is composed of cycles of length $\leq 2$ and its corresponding term is equal to 1, then its cycles correspond to non-adjacent edges in $G$, i.e. $\sigma$ defines a unique matching $\mu$ in $G$. Conversely, every matching $\mu$ in $G$ uniquely determines such a permutation $\sigma$. Consequently,
    \[
        \det(M_G) = m(G)\pmod2.
    \] 
    As $M_G \bf{1}_n=\bf{1}_n$ implies the graph is even, the theorem statement follows.
\end{proof}

\begin{remark}
    The proof of the previous theorem shows that a graph has a unique solution to the Lights Out game if and only if it has an odd number of matchings. The even condition guarantees that hitting all the vertices is a solution. 
\end{remark}

\section{Counting extremal graphs}\label{sec-count}

A natural starting point in the study of extremal graphs is to determine how many such graphs exist on $n$ vertices. We answer this question exactly.
Precisely, in this section, we count the number of labeled extremal graphs on the vertex set $[n]$. To do so, we make use of the following result.

\begin{lemma}\cite{macwilliams1969orthogonal}
    Let $\S^n$ denote the set of $n\times n$ symmetric invertible matrices over $\FF_2$. Then the cardinality $|\S^n|$ is given by
    \[
        |\S^n| = 2^{\binom{n+1}{2}} \prod_{i=1}^{\left\lfloor (n+1)/2 \right\rfloor}\left( 1 - 2^{\,1-2i} \right).
    \]
\end{lemma}

Recall that $G$ is extremal if and only if $A_G = M_G + I_n$ where $M_G$ is a symmetric invertible matrix with diagonal $\diag(M)=\mathbf{1}_n$ and $M_G\bf{1}_n = \bf{1}_n$. Therefore, we need to determine the cardinality of the following set:
\[\E^n = \{M \in \S^n\colon M\mathbf{1}_n = \mathbf{1}_n, \diag (M) = \mathbf{1}_n\}.\]

\begin{theorem}\label{thm:count}
    There is a bijection from the set of labeled extremal graphs on $n$ vertices to the set of $(n-2)\times(n-2)$ symmetric matrices $\S^{n-2}$. In particular,
    \[
        |\E^n|  = 2^{\binom{n-1}{2}} \prod_{i=1}^{\left\lfloor (n-1)/2 \right\rfloor}\left( 1 - 2^{\,1-2i} \right).
    \]
\end{theorem}

\begin{remark}
    Since the number of even labeled graphs on $n$ vertices is $2^{\binom{n-1}{2}}$, the proportion of even graphs that are extremal tends to $\sim 0.419$.
\end{remark}

\begin{proof}
    Let $\B^{n-1}=\{B\in \S^{n-1}\colon B\Ones_{n-1} = \diag(B)\}$. We construct a bijection $\Phi\colon \E^n \to \S^{n-2}$ by constructing two bijections $\Phi_1: \E^n \to \B^{n-1}$ and $\Phi_2: \B^{n-1} \to \S^{n-2}$.
    
    Let us define $\Phi_1$. Let $M \in \E^n$ and write 
    \[
        M = \left( \begin{array}{c|c} 1 & r^T \\\hline r & N \end{array} \right),
    \]
    where $r\in \FF_2^{n-1}$ and $N\in \FF_2^{(n-1)\times (n-1)}$. The condition $M\Ones_n = \Ones_n$ implies that $r^T \Ones_{n-1} = 0$ and $r + N \Ones_{n-1} = \Ones_{n-1}$.
    A standard Gaussian reduction leads to
    \[
        L A L^T = \left( \begin{array}{c|c} 1 & 0 \\\hline 0 & N+rr^T \end{array} \right)\text{, where } L = \left( \begin{array}{c|c} 1 & 0 \\\hline r & I_{n-1} \end{array} \right).
    \]
    The matrix $B=  N + rr^T$ satisfies the following properties:
    \begin{enumerate}
        \item $\det(B)=1$,
        \item $\diag(B) = \diag(N) + r = \Ones_{n-1} + r$,
        \item $B\Ones_{n-1} = (N+rr^T)\Ones_{n-1} = (\Ones_{n-1}+r) + 0 = \Ones_{n-1} + r$,
    \end{enumerate}
    and thus $B\in \mathcal{B}^{n-1}$. The map $\Phi_1: M \mapsto B$ is a bijection, as $r$ is uniquely determined as $r = \diag(B) + \Ones_{n-1}$.
    
    Now let us define $\Phi_2$. Let $B \in \mathcal{B}^{n-1}$ and write 
    \[
        B = \left( \begin{array}{c|c} \alpha & y^T \\\hline y & X \end{array}\right),
    \]
    where $X\in \S^{n-2}$. The condition $B\Ones_{n-1} = \diag(B)$ implies that $y + X\Ones_{n-1} = \diag(X)$. We apply the transformation 
    \[
        P = \left(\begin{array}{c|c} 1 & \Ones_{n-2}^T \\\hline 0 & I \end{array}\right)
    \]
    to $B$ to get
    \[
        B' = P B P^T = \left(\begin{array}{c|c} \gamma & (y+X\Ones_{n-2})^T \\\hline y+X\Ones_{n-2} & X \end{array}\right) = \left(\begin{array}{c|c} \gamma & \diag(X)^T \\\hline \diag(X) & X \end{array}\right) 
    \]
    for some scalar $\gamma$. We claim $X$ is invertible. Suppose for contradiction that $\det(X)=0$. Then there exists a non-zero $u$ such that $Xu=0$. For any $v\in \FF_2^{n-2}$ we have $v^T X v =\diag(X)^T v$ since $X$ is symmetric. Consequently, $\diag(X)^T u = u^T X u = 0$. The vector $(0, u^T)^T$ then lies in the kernel of $B'$, contradicting $\det(B')=\det(B)=1$. We define $\Phi_2: \B^{n-1} \to \S^{n-2}$ by $B\mapsto X$.
    
    Let us prove $\Phi_2$ is a bijection by exhibiting an inverse. Given $X \in \S^{n-2}$, we can uniquely reconstruct $B$. Let $\gamma = 1 + \diag(X)^T X^{-1} \diag(X)$ and let 
    \[
        B' \coloneqq \left( \begin{array}{c|c} \gamma & \diag(X)^T \\\hline \diag(X) & X \end{array}\right),
    \]
    where our choice of $\gamma$ makes $B'$ invertible. Now consider $$B \coloneqq P^{-1} B' P^{-T}= \left(\begin{array}{c|c} \alpha & y^T \\\hline y & X \end{array}\right),$$ where $\alpha = \gamma + \Ones_{n-2}^T X \Ones_{n-2}$ and $y = \diag(X) + X\Ones_{n-2}$. By construction, $B$ is symmetric and our choice of $\gamma$ ensured that $B$ is invertible. It remains to show that the defined $B$ satisfies the condition $B\Ones_{n-1} = \diag(B)$. Observe that $y + X\Ones_{n-2} = \diag(X) + X\Ones_{n-2} + X\Ones_{n-2} = \diag(X)$, so $(B\Ones_{n-1})_i = \diag(B)_i$ for $i \geq 2$. Lastly, we need to show that $\alpha + y^T \Ones_{n-2} = \alpha$, so it suffices to show that $y^T \Ones_{n-2} = \diag(X)^T \Ones_{n-2} + \Ones_{n-2}^T X \Ones_{n-2} = 0$, but this holds since $X$ is symmetric. Therefore, $B \in \B^{n-1}$. This defines an inverse to $\Phi_2$, hence it is a bijection. 
\end{proof}

In the remainder of this paper, we establish several structural properties of extremal graphs.

\section{$(0\bmod3)$-cycle free graphs}\label{sec-0cycle}

In this section, we establish a sufficient condition for an even graph to be extremal by restricting the length of its cycles. We begin by classifying when a cycle, the simplest even graph, is extremal.

\begin{theorem}\label{thm:cycles}
A cycle is extremal if and only if it does not have length divisible by $3$.
\end{theorem}
\begin{proof}
Let us denote the vertices of $C_k$ by $v_1,\dots,v_k$ such that $v_iv_j\in E(C_k)$ if $i-j = \pm 1 \pmod{k}$.
    Note that for any hitting set $S$,
    \begin{enumerate}
        \item $S$ needs to be dominating,
        \item if $v_i,v_{i+2} \in S$ then $v_{i+1}\in S$ (where the indices are modulo $k$), and
        \item if $v_i,v_{i+1} \in S$ then $v_{i+2}\in S$ (where the indices are modulo $k$).
    \end{enumerate}
If $|S|\neq k$, then (3) implies that $S$ is a stable set. Moreover, (1) and (2) imply that if $v_j\notin S$ then $|N(v_j)\cap S| =1$. By combining these remarks and centering balls of radius one at the vertices in $S$, there must exist a packing of balls that is a cover of $C_k$. Since each balls cover exactly $3$ vertices, we must have that $k$ is divisible by $3$. If $k$ is divisible by $3$ then $S=\set{v_2,v_5,\dots,v_{k-1}}$ is a hitting set.
\end{proof}

\cref{thm:cycles} seems to indicate that cycles with length divisible by $3$ are at the crux of what can make an even graph non-extremal. The following theorem captures this idea.

\begin{theorem}\label{thm:zm3c}
    If $G$ is an even graph with no cycle of length divisible by $3$, then $G$ is extremal.
\end{theorem}

To prove Theorem \ref{thm:zm3c}, we will make use of the following result.

\begin{lemma}\cite[Corollary 2]{GH83}\label{lemma:altcyc}
    Let $G$ be an even graph whose edges are colored red and blue so that every vertex is incident with an odd number of edges of each color. Then $G$ has an alternating cycle.
\end{lemma}

\begin{proof}[Proof of \cref{thm:zm3c}]
    Assume for contradiction that $G$ is not extremal. Then there exists a proper hitting set $S\subsetneq V(G)$ of $G$. Let $T = V(G)\setminus S$ be the nonempty set of vertices which are not hit. As $G$ is an even graph and $S$ is a hitting set for $G$, for any $s\in S$ we have $|N(s)\cap S|$ is even, and hence $|N(s)\cap T| = \deg(s)-|N(s)\cap S|$ is also even. Conversely, for any $t\in T$ we have $|N(t)\cap S|$ is odd, and hence $|N(t)\cap T|=\deg(t)-|N(t)\cap S|$ is also odd.

    Let $H'$ be the multigraph obtained from $G$ by adding a clique of new edges on $N(s)\cap T$ for every $s\in S$, and let $H = H'\setminus S$. Color an edge $e\in E(H)$ red if $e\in E(H')\setminus E(G)$ and blue if $e \in E(H')\cap E(G)$. Consider some arbitrary $v\in V(H)$. Since $v\in V(H) = T$, we have that $|N_G(v)\cap S|$ is odd and $|N_G(s)\cap T|$ is even for every $s\in N_G(v)\cap S$. Therefore, $v$ is incident with an odd number
    \[
        \sum_{s\in N_G(v)\cap S}(|N_G(s)\cap T|-1)
    \]
    of red edges in $H$. Furthermore, $v$ is incident with an odd number $|N_G(v)\cap T|$ of blue edges in $H$. Therefore, $H$ is an even graph whose edges are colored red and blue so that every vertex is incident with an odd number of edges of each color, so by \cref{lemma:altcyc} $H$ has an alternating red/blue cycle. Among all such cycles, choose $C$ to be the shortest.

    Let $C = v_1,v_2,\dots,v_k,v_1$ for some even $k\ge 2$, and assume without loss of generality that $v_1v_2$ is a red edge. By the definition of $H$, for any red edge $v_iv_{i+1}\in E(C)$, $i\in[k]$, there exists some $s\in S$ such that $v_i,v_{i+1}\in N_G(s)\cap T$. On the other hand, for any $s\in S$, at most one red edge of $C$ has both endpoints in $N_G(s)\cap T$. To see this, assume for contradiction that there exists some $s\in S$ and distinct $i,j\in[k]$ such that $v_iv_{i+1}$ and $v_jv_{j+1}$ are both red edges in $C$ and $v_i, v_{i+1}, v_j, v_{j+1}\in N_G(s)\cap T$. Assume without loss of generality that $i < j$. Since $C$ is an alternating cycle, $j\neq i+1$, and since $H[N_G(s)\cap T]$ is a clique, $v_iv_{j+1}\in E(H)$. But now the cycle $C' = v_1,v_2,\dots,v_i,v_{j+1},\dots,v_k,v_1$ is a shorter alternating cycle in $H$, contradicting the minimality of $C$. Therefore, for any red edge $v_iv_{i+1}\in E(C)$, $i\in[k]$, there exists a unique $s_{i,i+1}\in S$ such that $v_i,v_{i+1}\in N_G(s_{i,i+1})\cap T$. Furthermore, by the definition of $H$, any blue edge in $C$ corresponds to a unique edge of $G[T]$. 
    
    Lifting $C$ back to $G$, we have that
    \[
        C'' = v_1,s_{1,2},v_2,v_3,s_{3,4},v_4,\dots, v_{k-1}, s_{k-1, k}, v_kv_1
    \]
    is a cycle of length $3k/2$ in $G$, contradicting that $G$ has no cycle of length divisible by $3$.
\end{proof}

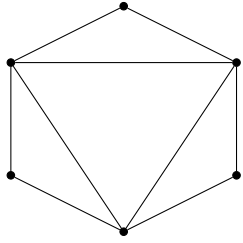
\begin{figure}[h!]
    \centering
    \scalebox{1}{\begin{tikzpicture}
	\begin{pgfonlayer}{nodelayer}
		\node [style=vertex] (0) at (0, 3) {};
		\node [style=vertex] (1) at (3, 1.5) {};
		\node [style=vertex] (2) at (-3, 1.5) {};
		\node [style=vertex] (3) at (-3, -1.5) {};
		\node [style=vertex] (4) at (3, -1.5) {};
		\node [style=vertex] (5) at (0, -3) {};
		\node [style=none] (6) at (3.5, 1.5) {};
		\node [style=none] (7) at (3.5, -1.5) {};
		\node [style=none] (8) at (0, -3.5) {};
		\node [style=none] (9) at (-3.5, -1.5) {};
		\node [style=none] (10) at (-3.5, 1.5) {};
		\node [style=none] (11) at (0, 3.5) {};
	\end{pgfonlayer}
	\begin{pgfonlayer}{edgelayer}
		\draw (1) to (4);
		\draw (4) to (5);
		\draw (5) to (3);
		\draw (3) to (2);
		\draw (2) to (0);
		\draw (0) to (1);
		\draw (2) to (5);
		\draw (5) to (1);
		\draw (1) to (2);
	\end{pgfonlayer}
\end{tikzpicture}}
    \caption{An extremal graph containing cycles of length divisible by $3$.}
    \label{fig: c6 and c3}
\end{figure}

Unfortunately, the converse of \cref{thm:cycles} does not hold. One can check that a $C_6$ with a $C_3$ inside, depicted in Figure \ref{fig: c6 and c3}, is an extremal graph. In fact, \cref{lemma:sun} can be used to verify that this graph is indeed extremal.

\section{Operations}\label{sec-operations}

To better understand graphs whose extremality does not follow from \cref{thm:zm3c}, we explore ways to build extremal graphs from smaller ones. In this section, we show that the family of extremal graphs is closed under several simple operations.

Clearly, the disjoint union of two graphs $G_1$ and $G_2$ is extremal if and only if both $G_1$ and $G_2$ are extremal. The following lemma shows that the $1$-join of two extremal graphs remains extremal.

\begin{lemma}[$1$-join]\label{lemma:1sum}
    Let $G_1$ and $G_2$ be extremal graphs, and let $x\in V(G_1)$ and $y \in V(G_2)$. Then the graph $G$ obtained by identifying $x$ and $y$ is extremal.
\end{lemma}
\begin{proof}
    Let $v\in V(G)$ be the vertex corresponding to the identification of $x$ and $y$. Assume for sake of contradiction that $G$ is not extremal, and let $S\subsetneq V(G)$ be a proper hitting set of $G$.  Let $a = |N(v)\cap S\cap V(G_1)|$ and $b = |N(v)\cap S\cap V(G_2)|$. 
    
    Case 1: $v\notin S$. Since $S$ is a hitting set of $G$, we have that $a + b$ is odd. Assume without loss of generality that $a$ is odd and $b$ is even. Then $S\setminus(V(G_2)\setminus\{v\})$ is a proper hitting set of $G_1$, contradicting that $G_1$ is extremal. 
    
    Case 2: $v\in S$. Since $S$ is a hitting set of $G$, we have that $a + b$ is even. First assume both $a$ and $b$ are even. Since $S$ is a proper hitting set of $G$, there exists some $u\in V(G)\setminus S$. Assume without loss of generality that $u\in V(G_1)$. Then $S\setminus(V(G_2)\setminus\{v\})$ is a proper hitting set of $G_1$, contradicting that $G_1$ is extremal. Now assume both $a$ and $b$ are odd. Let $H = G[S\cap V(G_1)]$. Then $\deg_H(v) = a$ is odd and $\deg_H(w)$ is even for any $w\in V(H)\setminus\{v\}$. Therefore, $H$ is a graph with exactly one odd-degree vertex, a contradiction.
\end{proof}

By the lemma above, taking the $1$-sum of two extremal graphs produces an extremal graph with a cut-vertex. Conversely, it is also true that any extremal graph with a cut-vertex $v$ is the $1$-join of two extremal graphs at $v$.

\begin{lemma}[Cut-vertex]\label{lemma:cutvtx}
    Let $G$ be an extremal graph with a cut-vertex $v$. Let $A,B\subseteq V(G)$ such that $A\cup B = V(G)$, $A\cap B = \{v\}$, and there are no edges from $A\setminus B$ to $B\setminus A$. Then $G[A]$ and $G[B]$ are extremal.
\end{lemma}

We deduce Lemma \ref{lemma:cutvtx} from a corollary of the following result on the parity of matchings in even graphs. For a matching $M$, let $V(M)$ be the set of endpoints of the edges of $M$. The size of a matching is the number of edges in it. For a graph $G$ and $A\subseteq V(G)$, we say that a matching $M$ \textit{covers} $A$ if $A\subseteq V(M)$.

\begin{theorem}\label{thm:oddset}
    Let $G$ be an even graph and $S\subseteq V(G)$ with $|S|$ odd. Then for any $k$, the number of size-$k$ matchings covering $S$ is even. 
\end{theorem}
\begin{proof}
    Let $G = (V, E)$, and work in the ring $R = \mathbb F_2[t][x_v:v\in V]/\{x_v^2:v\in V\}$. Define the polynomial $P = \prod_{ab\in E}(1+tx_ax_b)\in R$. For $A\subseteq V$, write $x^A = \prod_{a\in A}x_a$. By the definition of $R$, we have
    \[
        P = \sum_{M\text{ matching}}t^{|M|}x^{V(M)} = \sum_{i\ge 0}\sum_{\substack{U\subseteq V\\|U|=2i}} c_Ut^ix^U,
    \]
    where $c_U$ is the number of perfect matchings of $G[U]$ modulo $2$. Let $T = V\setminus S$ and define a second polynomial $Q = \prod_{w\in T}(1+x_w)\in R$. Consider the coefficient $[t^kx^V](P Q)$. Any $U\subseteq V$ with $|U| \neq 2k$ cannot contribute to this coefficient as its monomial misses $t^k$. Any $U\subseteq V$ with $U\not\supseteq S$ cannot contribute to this coefficient as its monomial misses $x_s$ for some $s\in S$. Any $U\subseteq V$ with $|U|=2k$ and $U\supseteq S$ contributes $c_U$ to this coefficient exactly once: choose ``$1$'' for $w\in T\cap U$ and ``$x_w$'' for $w\in T\setminus U$ when multiplying by $Q$. Therefore,
    \[
        [t^kx^V](P Q) = \sum_{\substack{U\subseteq V\\|U|=2k\\U\supseteq S}} c_U = m_k(G, S),
    \]
    where $m_k(G, S)$ is the number of size-$k$ matchings covering $S$ modulo $2$. For each $v\in V$, define the differential operator $\partial_v$ with $\partial_v(x_v) = 1$, $\partial_v(x_u) = 0$ for $u\neq v$, and $\partial_v(t) = 0$. Define $\partial = \sum_{v\in V}\partial_v$. For each $v\in V$, using the fact that $(1+tx_u x_v)^2 = 1$, the product rule gives
    \[
        \partial_v P = \partial_v\prod_{ab\in E}(1+tx_ax_b) = \sum_{u\in N(v)}tx_uP(1+tx_ux_v) = tP\sum_{u\in N(v)}x_u.
    \]
    Summing over all $v\in V$ and using $\deg(v) = 0$ modulo $2$ as $G$ is even,
    \[
        \partial P = tP\sum_{v\in V}\sum_{u\in N(v)}x_u = tP\sum_{v\in V}\deg(v)x_v = 0.
    \]
    Now consider the coefficient $[t^kx^V]\partial(P Q)$. Since $\partial$ reduces the degrees of all monomials by $1$, it must be that $[t^kx^V]\partial(P Q) = 0$. On the other hand, expanding $\partial(P Q)$ with the product rule and using $\partial P=0$,
    \[
        \partial(P Q) = (\partial P) Q + P(\partial Q) = P(\partial Q).
    \]
    Since $Q = \prod_{w\in T}(1+x_w)$, we have $\partial Q  = Q\sum_{w\in T}(1+x_w) = Q(|T| + \sum_{w\in T}x_w)$. Hence $\partial  (P Q) = P Q(|T| + \sum_{w\in T}x_w)$. Gathering coefficients, this gives
    \[
        [t^kx^V]\partial(P Q) = |T|m_k(G, S) + \sum_{w\in T}[t^kx^{V\setminus\{w\}}](P Q).
    \]
    Now fix $w\in T$ and consider the coefficient $[t^kx^{V\setminus\{w\}}](PQ)$. The same reasoning for determining $[t^k x^V](PQ)$ yields that any $U\subseteq V$ with $|U| \neq 2k$ or $U\not\supseteq S$ cannot contribute to this coefficient. Furthermore, any $U$ with $w\in U$ also cannot contribute, since it cannot have $x_w$ in its monomial. Finally, any $U\subseteq V$ with $|U|=2k$, $U\supseteq S$, and $U\not\ni w$ contributes $c_U$ to this coefficient exactly once: choose ``$1$'' for the vertices in $T\cap U$ and $w$, and choose $x_a$ for $a\in T\setminus U$ when multiplying by $Q$. Therefore, 
    \[
        [t^kx^{V\setminus\{w\}}](P Q) = \sum_{\substack{U\subseteq V\\|U|=2k\\U\supseteq S\\U\not\ni w}} c_U
    \]
    is the number of size-$k$ matchings covering $S$ but not covering $w$. Summing over all $w\in T$ and interpreting the scalars $|S|$, $|T|$, $|U|$, and $|V|$ modulo $2$,
        \begin{align*}
        \sum_{w\in T}[t^kx^{V\setminus\{w\}}](P Q) &= \sum_{w\in T}\sum_{\substack{U\subseteq V\\|U|=2k\\U\supseteq S\\U\not\ni w}} c_U = \sum_{w\in T}\sum_{\substack{U\subseteq V\\|U|=2k\\U\supseteq S}} \mathbbm{1}[w\notin U]c_U\\
        &= \sum_{\substack{U\subseteq V\\|U|=2k\\U\supseteq S}}c_U|T\setminus U| = \sum_{\substack{U\subseteq V\\|U|=2k\\U\supseteq S}}c_U(|V|-|U|)\\
        &= \sum_{\substack{U\subseteq V\\|U|=2k\\U\supseteq S}}c_U(|V|-2k) = |V|\sum_{\substack{U\subseteq V\\|U|=2k\\U\supseteq S}}c_U = |V|m_k(G, S).
    \end{align*}
    It follows that $0 = [t^kx^V]\partial(P Q) = |T|m_k(G, S) + |V|m_k(G, S) = (|V|-|T|)m_k(G,S) = |S|m_k(G,S)$. Since $|S|$ is odd, this means $m_k(G,S) = 0$, so the number of size-$k$ matchings covering $S$ is even.
\end{proof}

The following corollary is an immediate consequence of Theorem \ref{thm:oddset} with $|S|=1$ over all $k\ge 0$.

\begin{corollary}\label{cor:vtxcover}
    For any even graph $G$ and any $v\in V(G)$, the number of matchings covering $v$ is even.
\end{corollary}

We now use Corollary \ref{cor:vtxcover} to deduce Lemma \ref{lemma:cutvtx}. For a graph $G$ and a vertex $v\in V(G)$, let $m(G, v)$ denote the number of matchings in $G$ covering $v$.

\begin{proof}[Proof of Lemma \ref{lemma:cutvtx}.]
    Let $G_1 = G[A]$ and $G_2 = G[B]$. Suppose for contradiction that, without loss of generality, $G_1$ is not extremal. Since $G$ is an even graph, it must be that $G_1$ and $G_2$ are even graphs. Indeed, for $i\in [2]$ and $u \in V(G_i)\setminus\{v\}$, we have $\deg_{G_i}(u) = \deg_G(u)$ is even. This further implies that $\deg_{G_i}(v)$ is even, as otherwise $G_i$ would have exactly one vertex of odd degree, a contradiction.  
    
    As $G$ is extremal, by Theorem \ref{thm:oddmatchings} we have $m(G)$ is odd. Any matching in $G$ either covers $v$ via an edge in $G_1$, covers $v$ via an edge in $G_2$, or does not cover $v$. Since there are no edges from $A\setminus B$ to $B\setminus A$,
    \begin{align*}
        m(G) &= m(G_1\setminus v)\cdot m(G_2\setminus v) + m(G_1, v)\cdot m(G_2\setminus v) + m(G_1\setminus v)\cdot m(G_2, v)
        \\&= m(G_2\setminus v)[m(G_1\setminus v) + m(G_1, v)] + m(G_1\setminus v)\cdot m(G_2,v)
        \\&= m(G_2\setminus v)\cdot m(G_1) + m(G_1\setminus v)\cdot m(G_2, v).
    \end{align*}
    Since $G_1$ is not extremal, $m(G_1)$ is even, so $m(G)$ has the same parity as $m(G_1\setminus v)\cdot m(G_2, v)$. However, by Corollary \ref{cor:vtxcover} we have $m(G_2, v)$ is even, so $m(G)$ must also be even, a contradiction.
\end{proof}

Extremal graphs are not closed under subdivisions, as exemplified by the family of cycles. However, we do have the following.

\begin{lemma}[Triple subdivision]\label{triple subdivision}
    A graph $G$ is extremal if and only if for any edge $e\in E(G)$, the graph $G'$ obtained by subdividing $e$ three times is extremal.    
\end{lemma}
\begin{proof} In the following, all sums will be taken modulo $2$.
    Let $e=xy\in E(G)$ and let $G'$ be the graph obtained by subdividing $e$ three times, introducing the path $x\dd a\dd b\dd c\dd y$. Let us relate the total number of matchings in $G'$ with the one in $G$. 
    For $I\subseteq\{xa,ab,bc,cy\}$, we define $K_I$ as the number of matching $M$ on $G'$ such that $M\cap \{xa,ab,bc,cy\} = I$. We have that
    \begin{align*}
        m(G') &= K_\varnothing + K_{xa}+ K_{xa,bc} + K_{cy}+ K_{ab,cy} + K_{ab} +  K_{bc} + K_{xa,cy}\\
        &= K_\varnothing + K_{xa}+ K_{xa} + K_{cy}+ K_{cy} + K_{ab} +  K_{ab} + K_{xa,cy}\\
        &= K_\varnothing + K_{xa,cy}.
    \end{align*} 
    Notice that the $K_{xa,cy}$ is the number of matches on $G$ that includes $e$. Similarly, $K_\varnothing$ is equivalent to the number of matchings of $G$ not including $e$. Therefore, \begin{align*}
        m(G') &=K_\varnothing + K_{xa,cy}\\
        &= m(G-e) + m(G,e)\\
        &=  m(G).
    \end{align*}
    \cref{thm:oddmatchings} lets us conclude.
\end{proof}

Let $G_1$ and $G_2$ graphs. Let $X_1\subseteq V(G_1)$ and $X_2\subseteq V(G_2)$. We define the \textit{completion of $G_1,G_2$ along $X_1$,$X_2$} as the graph obtained by taking the disjoint union of $G_1$ and $G_2$ and adding all possible edges between $X_1$ and $X_2$. Formally, it is the graph $(V(G_1)\cup V(G_2), E(G_1)\cup E(G_2) \cup \{xy : x\in X_1, y\in X_2 \})$. See \cref{fig: even completion} for an illustration.
\begin{figure}[h]
    \centering
    \scalebox{1.2}{\begin{tikzpicture}
	\begin{pgfonlayer}{nodelayer}
		\node [style=none] (0) at (-4.25, -0.5) {};
		\node [style=none] (1) at (-3.5, 1.25) {};
		\node [style=none] (2) at (-0.5, -1) {};
		\node [style=none] (3) at (-3.75, -3) {};
		\node [style=none] (4) at (-1.75, -3.5) {};
		\node [style=none] (5) at (-1.75, -2.5) {};
		\node [style=none] (6) at (-1.75, -0.75) {};
		\node [style=none] (8) at (-0.75, 0.75) {};
		\node [style=none] (9) at (5, 1.25) {};
		\node [style=none] (10) at (7, 0.25) {};
		\node [style=none] (11) at (6.25, -3) {};
		\node [style=none] (12) at (4.75, -3.75) {};
		\node [style=none] (13) at (3.5, -2.25) {};
		\node [style=none] (14) at (3.5, -0.75) {};
		\node [style=none] (17) at (4.75, -2.5) {};
		\node [style=none] (19) at (4.75, -0.75) {};
		\node [style=none] (20) at (-2.5, 0.7) {$G_1$};
		\node [style=none] (21) at (-1.74, -1.6) {$X_1$};
		\node [style=none] (22) at (5, 0.7) {$G_2$};
		\node [style=none] (23) at (4.75, -1.6) {$X_2$};
	\end{pgfonlayer}
	\begin{pgfonlayer}{edgelayer}
		\draw [in=90, out=60, looseness=0.75] (1.center) to (8.center);
		\draw (8.center) to (2.center);
		\draw [in=45, out=-75, looseness=0.75] (2.center) to (4.center);
		\draw [in=330, out=-135, looseness=0.75] (4.center) to (3.center);
		\draw [in=-105, out=150] (3.center) to (0.center);
		\draw [in=240, out=75, looseness=0.50] (0.center) to (1.center);
		\draw [in=180, out=-180] (6.center) to (5.center);
		\draw [in=90, out=0, looseness=0.75] (9.center) to (10.center);
		\draw [in=60, out=-75, looseness=0.50] (10.center) to (11.center);
		\draw [in=-15, out=-120, looseness=0.75] (11.center) to (12.center);
		\draw [in=-90, out=165, looseness=0.75] (12.center) to (13.center);
		\draw [in=-105, out=90, looseness=0.75] (13.center) to (14.center);
		\draw [in=-180, out=75] (14.center) to (9.center);
		\draw [in=180, out=-180] (19.center) to (17.center);
		\draw [in=0, out=0] (6.center) to (5.center);
		\draw [in=0, out=0] (19.center) to (17.center);
		\draw (6.center) to (19.center);
		\draw (19.center) to (5.center);
		\draw (6.center) to (17.center);
		\draw (17.center) to (5.center);
	\end{pgfonlayer}
\end{tikzpicture}}
    \caption{Example of the 'even completion' operation.}
    \label{fig: even completion}
\end{figure}
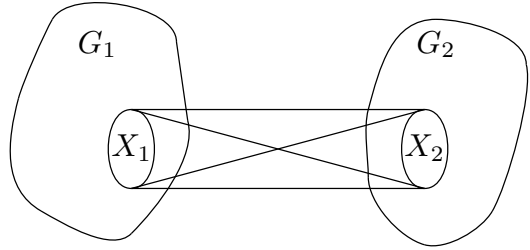

\begin{lemma}[Even completion]
    Let $G_1$, $G_2$ be graphs. Let $X_1\subseteq V(G_1)$ and $X_2\subseteq V(G_2)$ be sets of even size. Let $G$ be the completion of $G_1,G_2$ along $X_1$,$X_2$.
    Then $G$ is extremal if and only if $G_1$ and $G_2$ are extremal.
\end{lemma}

\begin{proof}
    Work modulo $2$ throughout. Clearly $G$ is even if and only if $G_1$ and $G_2$ are even. Let $H$ be the bipartite subgraph of $G$ with $V(H) = X_1\cup X_2$ and $E(H) = E(X_1,X_2)$ (the edges added in the completion operation). 
    Let $\mathcal M$ be the set of matchings of $H$. Since there are no edges from $V(G_1)$ to $V(G_2)$ except those in $E(H)$, 
    \[
        m(G) = \sum_{M\in \mathcal M} m(G_1\setminus V(M)) \cdot m(G_2\setminus V(M)).
    \]
    We will group the terms of this summation into three: the empty matching, matchings of size exactly one, and matchings of size at least two. The latter two will be shown to contribute zero to the sum. We have:
    \[
        m(G) = m(G_1)\cdot m(G_2) + \sum_{u\in X_1, v\in X_2} m(G_1\sm u) \cdot m(G_2\sm v) + \sum_{\substack{M\in \mathcal M\\ |M| \geq 2}} m(G_1\sm V(M)) \cdot m(G_2\sm V(M)).
    \]
    By \cref{cor:vtxcover}, for any $u\in X_1$, we have $m(G_1\setminus u) = m(G_1) - m(G_1,u) = m(G_1)$. Similarly, for any $v\in X_2$, we have $m(G_2\setminus v) = m(G_2)$. Therefore,
    \[
        m(G) = m(G_1)\cdot m(G_2)  + |E(X_1,X_2)| \cdot m(G_1)\cdot m(G_2) + \sum_{\substack{M\in \mathcal M\\ |M| \geq 2}} m(G_1\sm V(M)) \cdot m(G_2\sm V(M)).
    \]
    Since $|X_1|$ and $|X_2|$ are even, $|E(X_1,X_2)| = 0$, so this simplifies to
    \[
        m(G) = m(G_1)\cdot m(G_2)  + \sum_{\substack{M\in \mathcal M\\ |M| \geq 2}} m(G_1\sm V(M)) \cdot m(G_2\sm V(M)). 
    \]
    Consider some arbitrary $2\le k\le \min(|X_1|,|X_2|)$. Let $V_i\subseteq X_i$ be sets of size $k$ for $i\in[2]$. Then the number of matchings $M$ in $\mathcal M$ with $V(M) = V_1\cup V_2$ is exactly $k!$, which is $0$ modulo $2$. Grouping matchings by $|M|$,
    \[
        \sum_{\substack{M\in \mathcal M\\ |M| \geq 2}} m(G_1\sm V(M)) \cdot m(G_2\sm V(M)) =  \sum_{k=2}^{\min(|X_1|,|X_2|)} \sum_{\substack{V_1 \subseteq X_1,\\ V_2 \subseteq X_2,\\ |V_1|=|V_2|=k}} k! \cdot m(G_1\sm V_1) \cdot m(G_2\sm V_2) = 0.
    \]
    Therefore, $m(G) = m(G_1)\cdot m(G_2)$, so over $\mathbb F_2$ we have $m(G) = 1$ if and only if $m(G_1) = 1$ and $m(G_2)=1$. By \cref{thm:oddmatchings}, it follows that $G$ is extremal if and only if $G_1$ and $G_2$ are extremal. 
\end{proof}

We now define a \textit{sun cycle} operation on a graph $G$. Let $\{v_1,\ldots,v_k\} \subseteq V(G)$ (with $v_{k+1} = v_1$ for convenience). We construct a new graph $G'$ from $G$ as follows. For every $i \in [k]$,
\begin{itemize}
    \item If $(v_i,v_{i+1}) \in E(G)$, we remove $(v_i,v_{i+1})$ from $G'$, and if $(v_i,v_{i+1}) \not\in E(G)$, we add $(v_i,v_{i+1})$ to $G'$.
    \item Add an auxiliary vertex $w_i$ with edges just to $v_i$ and $v_{i+1}$.
\end{itemize}
See \cref{fig: sun} for an illustration.
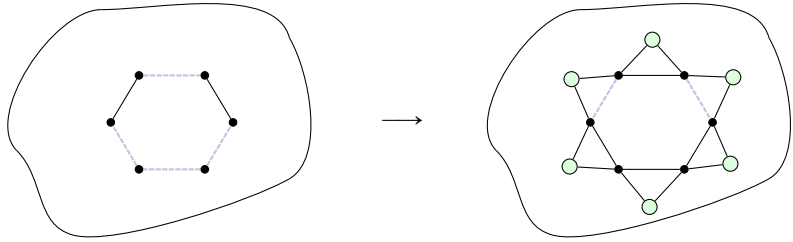
\begin{figure}[h]
    \centering
    \scalebox{1}{\begin{tikzpicture}
	\begin{pgfonlayer}{nodelayer}
		\node [style=none] (0) at (-7.25, -0.75) {};
		\node [style=none] (1) at (-5, 3) {};
		\node [style=none] (2) at (0, 2.25) {};
		\node [style=none] (3) at (0, -1.5) {};
		\node [style=none] (4) at (-5.75, -3) {};
		\node [style=vertex] (5) at (-4.75, 0) {};
		\node [style=vertex] (6) at (-4, 1.25) {};
		\node [style=vertex] (9) at (-4, -1.25) {};
		\node [style=vertex] (10) at (-1.5, 0) {};
		\node [style=vertex] (11) at (-2.25, 1.25) {};
		\node [style=vertex] (12) at (-2.25, -1.25) {};
		\node [style=none] (13) at (3, 0) {$\longrightarrow$};
		\node [style=none] (14) at (5.5, -0.75) {};
		\node [style=none] (15) at (7.75, 3) {};
		\node [style=none] (16) at (12.75, 2.25) {};
		\node [style=none] (17) at (12.75, -1.5) {};
		\node [style=none] (18) at (7, -3) {};
		\node [style=vertex] (19) at (8, 0) {};
		\node [style=vertex] (20) at (8.75, 1.25) {};
		\node [style=vertex] (21) at (8.75, -1.25) {};
		\node [style=vertex] (22) at (11.25, 0) {};
		\node [style=vertex] (23) at (10.5, 1.25) {};
		\node [style=vertex] (24) at (10.5, -1.25) {};
		\node [style=Z dot] (26) at (7.5, 1.15) {};
		\node [style=Z dot] (27) at (9.65, 2.2) {};
		\node [style=Z dot] (28) at (11.8, 1.2) {};
		\node [style=Z dot] (29) at (11.725, -1.1) {};
		\node [style=Z dot] (30) at (9.575, -2.25) {};
		\node [style=Z dot] (31) at (7.45, -1.175) {};
	\end{pgfonlayer}
	\begin{pgfonlayer}{edgelayer}
		\draw [in=180, out=135, looseness=0.75] (0.center) to (1.center);
		\draw [in=105, out=0, looseness=0.75] (1.center) to (2.center);
		\draw [in=30, out=-60, looseness=0.75] (2.center) to (3.center);
		\draw [in=-15, out=-150, looseness=0.50] (3.center) to (4.center);
		\draw [in=315, out=165] (4.center) to (0.center);
		\draw (6) to (5);
		\draw [style=box edge] (9) to (5);
		\draw (11) to (10);
		\draw [style=box edge] (12) to (10);
		\draw [style=box edge] (6) to (11);
		\draw [style=box edge] (12) to (9);
		\draw [in=180, out=135, looseness=0.75] (14.center) to (15.center);
		\draw [in=105, out=0, looseness=0.75] (15.center) to (16.center);
		\draw [in=30, out=-60, looseness=0.75] (16.center) to (17.center);
		\draw [in=-15, out=-150, looseness=0.50] (17.center) to (18.center);
		\draw [in=315, out=165] (18.center) to (14.center);
		\draw [style=box edge] (20) to (19);
		\draw [style=simple] (21) to (19);
		\draw [style=box edge] (23) to (22);
		\draw [style=simple] (24) to (22);
		\draw [style=simple] (20) to (23);
		\draw [style=simple] (24) to (21);
		\draw [style=simple] (26) to (20);
		\draw [style=simple] (26) to (19);
		\draw [style=simple] (31) to (19);
		\draw [style=simple] (31) to (21);
		\draw [style=simple] (30) to (21);
		\draw [style=simple] (30) to (24);
		\draw [style=simple] (29) to (24);
		\draw [style=simple] (29) to (22);
		\draw [style=simple] (28) to (22);
		\draw [style=simple] (28) to (23);
		\draw [style=simple] (23) to (27);
		\draw [style=simple] (20) to (27);
	\end{pgfonlayer}
\end{tikzpicture}}
    \caption{Example of the 'sun cycle' operation.} 
    \label{fig: sun}
\end{figure}

\begin{lemma}[Sun cycle]\label{lemma:sun}
    Let $G$ be a graph and let $S = \{v_1,\ldots,v_k\} \subseteq V(G)$. Let $G'$ be formed from $G$ by applying the sun cycle operation using $S$. Then $G$ is extremal if and only if $G'$ is extremal.
\end{lemma}
\begin{proof}
    ($\implies$) Suppose $G$ is extremal. Suppose by contradiction that $G'$ admits a proper hitting set $S' \subseteq V(G')$. We claim that $S = S' \setminus \{w_1,\ldots,w_k\} \subseteq V(G)$ is a hitting set for $G$.  
    
    Consider any $v_i$ in $G'$ and suppose $v_i$ is hit. Consider the pair of vertices $\{v_{i-1},w_{i-1}\}$. To ensure that $w_{i-1}$ is toggled, it must be the case that either both or neither of them are hit. The same is true for the pair $\{w_i,v_{i+1}\}$. On reversing the sun-cycle operation to obtain $G$, it is easy to see that the parity of $v_i$ is preserved.
    
    Now suppose $v_i$ is not hit. To ensure that $w_{i-1}$ and $w_i$ are toggled, it must be the case that exactly one of $\{v_{i-1},w_{i-1}\}$ and exactly one of $\{w_i,v_{i+1}\}$ has been hit. Reversing the sun-cycle operation to obtain $G$, the parity of $v$ is again preserved. 
    
    So it follows that $S$ is a hitting set for $G$. Since $G$ is extremal, it must be the case that $S = V(G)$. Since $S'$ is a hitting set for $G'$ and each $w_i$ has exactly two hit neighbors in $V(G)$, it must be the case that $w_i$ is also hit. It follows that $S' = V(G')$, a contradiction.
    
    ($\impliedby$) Suppose $G'$ is extremal. Suppose for contradiction that $G$ admits a proper hitting set $S$. Then set $S' = S \cup \{w_i \colon v_i, v_{i+1} \in S\}$. By similar parity arguments as the previous direction, $S'$ is a hitting set for $G'$, which implies that $S' = V(G')$ and in turn forces $S = V(G)$, a contradiction.
\end{proof}

Unfortunately, the operations presented in this section are not complete: not all extremal graphs can be generated from these operations starting from a set of `simple' extremal graphs like the set of graphs with no cycle of length $0$ mod $3$. It would be interesting to find a complete structural characterization of the extremal graphs. 

\textbf{Acknowledgments:} We would like to thank Noga Alon, Matija Bucić, Chayim Lowen, and Paul Seymour for helpful discussions.

\bibliography{ref}
\bibliographystyle{amsalpha}
\end{document}